\documentclass{amsart}
\usepackage{amssymb}
\usepackage{fontenc}
\usepackage{url}
\usepackage{hyperref}
\newtheorem{conjecture}{Conjecture}

\newtheorem{theorem}{Theorem}

\title{Blow-up trick in combinatorics}
\author{Veronica Phan*}
\thanks{*Ho Chi Minh City; email: \url{kyubivulpes@gmail.com}}
\begin{document}
\begin{abstract}
Blow-up in graph theory is a procedure in which each vertex is replaced by copies of itself, and two copies are adjacent if and only if the original vertices are adjacent. In this paper, we extend the concept of graph blow-up to a more general combinatorial context and discuss its potential applications.
\end{abstract}
\maketitle
\section{Blow-up of a graph}

Given a graph $G=(V,E)$, a blow-up of $G$ is obtained by replacing each vertex $v\in V$ with an independent set of vertices $I_v$ ($I_v$ could be empty), and for every edge $uv\in E$, adding all possible edges between $I_u,I_v$. More formally, the resulting graph is $$G'=\left(\bigcup_{v\in V}I_v,\bigcup_{uv\in E}I_u\times I_v\right)$$.

What is special about a blow-up of a graph? From the construction, we see that the blow-up of a graph can preserve certain local properties of the original graph, such as the clique number $\omega(G)$. This observation is reminiscent of Turán's theorem, but can this simple procedure help us solve it?

\begin{theorem}[Turán's theorem, simple form]\label{turan}
Every graph with $n$ vertices that does not contain $K_{r+1}$ as a subgraph (or $K_{r+1}$-free) has no more than $\left(1-\frac{1}{r}\right)\frac{n^2}{2}$ edges.
\end{theorem}

Assume $G=(V,E)$ is a $K_{r+1}$-free graph with $|V|=n$. We try to blow-up $G$ by replacing each vertex $v\in V$ with an independent set $I_v$ of size $|I_v|=c_v$. The resulting graph has $\sum_{v\in V}c_v$ vertices and $\sum_{uv\in E}c_uc_v$ edges. By Turán's theorem, we expect that $$\sum_{uv\in E}c_uc_v\leq\left(1-\frac{1}{r}\right)\frac{\left(\sum_{v\in V}c_v\right)^2}{2}$$.

Let $x_v=\frac{c_v}{\sum_{v\in V}c_v}$, then $\sum_{v\in V}x_v=1,0\leq x_v\leq1,\forall v\in V$. Divide both sides of the inequality by $\left(\sum_{v\in V}c_v\right)^2$, we obtain $$\sum_{uv\in E}x_ux_v\leq\frac{1}{2}\left(1-\frac{1}{r}\right)$$

This is precisely the key inequality T. S. Motzkin and E. G. Straus \cite{Motzkin-Straus} used to prove Turán's theorem.

\begin{theorem}[Motzkin-Straus theorem]\label{M-S}
Let $G=(V,E)$ be a $K_{r+1}$-free graph. We assign each vertex $v\in V$ a nonnegative real number $x_v$ such that $\sum_{v\in V}x_v=1$. Then we have $$\sum_{uv\in E}x_ux_v\leq\frac{1}{2}\left(1-\frac{1}{r}\right)$$
\end{theorem}

\begin{proof}[Proof (sketch)]
We consider maximizing $\mathcal{L}_G(x_v)_{v\in V}=\sum_{uv\in E}x_ux_v$. Assume $a,b$ are not adjacent in $G$, define $G_a,G_b$ be the graphs created from $G$ by removing the vertex $b,a$ from $V$, respectively. We assign nonnegative number $(x^{a}_v)_{V\setminus\{b\}}$ to vertices of $G_a$ as follow: $x^a_a=x_a+x_b, x^a_v=x_v, v\neq a$, similarly for $G_b$. Then $$(x_a+x_b)\mathcal{L}_G(x_v)_{v\in V}=x_a\mathcal{L}_{G_a}(x^a_v)_{v\in V\setminus\{b\}}+x_b\mathcal{L}_{G_b}(x^b_v)_{v\in V\setminus\{a\}}$$

Thus, we may remove either vertex $a$ or $b$ so that the value of $\mathcal{L}_G(x_v)_{v\in V}$ does not decrease. Repeating this process until every pair of remaining vertices is adjacent, we eventually obtain a complete graph on at most $r$ vertices. The result then follows straightforwardly.
\end{proof}

\begin{proof}[Proof of theorem \ref{turan}]
We assign each vertex $v\in V$ number $\frac{1}{n}$, the result follows by theorem \ref{M-S}.
\end{proof}

We have just proved the simple form of Turán's theorem by blowing up the graph, deriving the key inequality in a natural way, and completing with a simple trick. So why does the blow-up trick work?

$-$ Blow-up preserves the important local properties of the object.

$-$ Blow-up transforms a discrete problem into a semi-continuous one, allowing us to apply tools from analysis.

$-$ By blowing up, we can vary the size of the object, which helps us relate object of different size.

\section{Blow-up of hypergraph}

Now we extend the blow-up trick to the hypergraph setting, specifically to $3$-graph. Let us begin with Turán conjecture \cite{Turan} on $K^3_4$-free $3$-graph.

\begin{conjecture}[Turán conjecture, simple form]
Every $K^3_4$-free $3$-graph with $n$ vertices has no more than $\frac{5}{54}n^3$ edges
\end{conjecture}

Similarly to the graph case, given a $K^3_4$-free $3$-graph $G=(V,E)$, we replace each vertex $v\in V$ with an independent set of vertices $I_v$ of size $|I_v|=c_v$, and for every edge $(u,v,w)\in E$, we add all possible $3$-edges between $I_u,I_v$ and $I_w$. We then seek to maximize $\sum_{(u,v,w)\in E}x_ux_vx_w$ over all nonnegative real numbers $(x_v)_{v\in V}$ statisfying $\sum_{v\in V}x_v=1$.

But now the problem becomes more complex become the expression is nonlinear. Moreover, $(u,v,w)\notin E$ does not exclude the possibility that there exists some $t\in V$ such that $(u,v,t)\in E$, which make the calculation even more complicated.

So what is the problem here? We have not added enough information to the blow-up. We need to add $3$-edges of the form $(u_1,u_2,v_1)$ for $u_1,u_2\in I_u$ and $v_1\in I_v$, which can be represented as a directed $2$-edge $u\rightarrow v$. However, we cannot add edges of this type arbitrarily, since the resulting $3$-graph may no longer be $K^3_4$-free.

Suppose we have the original $K^3_4$-free $3$-graph $G=(V,E)$ and the directed graph $D=(V,A)$ describing how the special $3$-edges are added into the blow-up construction. We now determine the forbidden configurations that ensure the resulting graph remains $K^3_4$-free:

$-$ If $u\rightarrow v,v\rightarrow u\in A$, then the vertices $\{u_1,u_2,v_1,v_2\}$ where $u_1,u_2\in I_u$ and $v_1,v_2\in I_v$ induced a copy of $K^3_4$. Therefore, $D$ must be an oriented graph (in particular, it contains no pair of opposite directed edges).

$-$ If $u\rightarrow v,u\rightarrow w\in A, (u,v,w)\in E$, then the vertices $\{u_1,u_2,v,w\}$ where $u_1,u_2\in I_u$ induced a copy of $K^3_4$. Therefore if $(u,v,w)\in E$, then at most one of $u\to v$ and $u\to w$ may belong to $A$. Equivalently, whenever $u\rightarrow v,u\rightarrow w\in A$, we must have $(u,v,w)\notin E$.

The problem is still complicated, but it now has much more structure. To simplify it (at the cost of no longer addressing the original question directly), we fixed the oriented graph $D=(V,A)$ and add $3$-edge in a suitable way that maximizes the number of edges while still keep the resulting $3$-graph $K^3_4$-free. A natural construction is the following: $(u,v,w)\in E$ if and only if the vertices $\{u,v,w\}$ induced a subgraph of $D$ with at least $2$ edges and containing no copy of the directed graph $S=(\{a,b,c\},a\rightarrow b,a\rightarrow c)$.

The only missing condition for the resulting $3$-graph $G=(V,E)$ to be $K^3_4$-free is that $D$ does not contain directed cycle of length $4$. This is essentially the beautiful construction of $K^3_4$-free $3$-graph due to Fon-der-Flaass \cite{FDF}, we recover it simply by blowing up the $3$-graph and observing what happens.

Now we simplify the problem further by ignoring the condition that $D$ does not contain a directed cycle of length $4$, at the cost that the resulting graph may no longer $K^3_4$-free. Given a oriented graph $D=(V,A)$ and $3$-graph $G=(V,E)$ constructed as above, we seek to maximize $$\frac{\sum_{(u,v,w)\in E}c_uc_vc_w+\sum_{u\rightarrow v\in A}\frac{c_u(c_u-1)}{2}c_v}{\left(\sum_{v\in V}c_v\right)^3}$$ over all nonnegative integer $(c_v)_{v\in V}$

Let $x_v=\frac{c_v}{\sum_{v\in V}c_v}$ and suppose that the $c_v$ are arbitrary large, then the above expression becomes asymptotically $$\sum_{(u,v,w)\in E}x_ux_vx_w+\frac{1}{2}\sum_{u\rightarrow v\in A}x_u^2x_v$$

Now we are able to obtain the bound $\frac{3}{32}$ for the above expression. This is not the desired bound of $\frac{5}{54}$ since we have removed an important condition. Nevertheless, it is still equivalent to a strong result of Alexander A. Razborov \cite{AAR} stating that the edge density of a Fon-der-Flaass graph is at least $\frac{7}{16}(1-o(1))$.

We generalize our previous construction as follow: Instead of an oriented graph $D=(V,A)$, we allow both $u\rightarrow v,v\rightarrow u$ to be edges of $D$ but each assigned a nonnegative weight $p_{u\rightarrow v}$ and $p_{v\rightarrow u}$ respectively. Equivalently, we may view this as an undirected edge $(u,v)$ with two directional weights. Under the about condition, we should have $p_{u\rightarrow v}+p_{v\rightarrow u}\leq1$. For simplify, we further assume that $p_{u\rightarrow v}+p_{v\rightarrow u}=1$. We can see that the previous construction is just a special case where $p_{u\rightarrow v}=1$ and $p_{v\rightarrow u}=0$ for each $u\rightarrow v\in A$.

Relax the condition: $(u,v,w)\in E$ if and only if the vertices $\{u,v,w\}$ induced a subgraph of $D$ with at least $2$ edges.View $p_{u\rightarrow v}$ as the probability that $u\rightarrow v\in A$. Then the probability of $(u,v,w)\in E$ is $0$ if the subgraph of $D$ induced by $\{u,v,w\}$ contains fewer than $2$ edges, and otherwise it is $1-p_{u\rightarrow v}p_{u\rightarrow w}-p_{v\rightarrow w}p_{v\rightarrow u}-p_{w\rightarrow u}p_{w\rightarrow v}$ (if $(u,v)\notin A$, we treat $p_{u\rightarrow v}=p_{v\rightarrow u}=0$). So finally, we seek to maximize $$\sum_{(u,v,w)\in E}(1-p_{u\rightarrow v}p_{u\rightarrow w}-p_{v\rightarrow w}p_{v\rightarrow u}-p_{w\rightarrow u}p_{w\rightarrow v})x_ux_vx_w+\frac{1}{2}\sum_{(u,v)\in A}(p_{u\rightarrow v}^2x_u^2x_v+p_{v\rightarrow u}^2x_v^2x_u)$$ over all nonnegative $(x_v)_{v\in V},p_{u\rightarrow v}$ that satisfy $\sum_{v\in V}x_v=1$ and $p_{u\rightarrow v}+p_{v\rightarrow u}=1$ if $(u,v)\in A$ and $0$ if otherwise.

After some manipulations, the above expression is no more than $$\sum_{(u,v,w)\in E}x_ux_vx_w+\frac{1}{2}\sum_{(u,v)\in A}(x_u^2x_v+x_v^2x_u)-\frac{1}{2}\sum_{u\in V}x_u\left(\sum_{(u,v)\in A}p_{u\rightarrow v}x_v\right)^2$$

Note that $\sum_{v\in V}x_v=1$ and $\sum_{v\in V}\left(\sum_{(u,v)\in A}p_{u\rightarrow v}x_v\right)=\sum_{(u,v)\in A}x_ux_v$, apply Cauchy-Schward to the third term we obtain $$\sum_{(u,v,w)\in E}x_ux_vx_w+\frac{1}{2}\sum_{(u,v)\in A}(x_u^2x_v+x_v^2x_u)-\frac{1}{2}\left(\sum_{(u,v)\in A}x_ux_v\right)^2$$

This magical Lagrangian is the key to bounding the number of edges of the Fon-der-Flaass graph. Using the same trick as in the Motzkin-Straus, we obtain our desired result in an elementary way, for more detail, see \cite{verox}.

So by blowing up the $K^3_4$-free $3$-graph and carrying out a careful analysis, we recover the same result as Alexander A. Razborov \cite{AAR} without heavily relying on the flag algebra method. Furthermore, we can see several remaining problems that need to be resolved in order to reach the optimal $\frac{5}{54}n^3$ bound, such as making use of the condition that $D$ does not contain a directed cycle of length $4$, or adding more structure to each edge of $D$, and so on.

\section{Blow-up trick in combinatoric}
The blow-up trick is quite useful in graphs and hypergraphs, but what about in a more general combinatorial setting? Let’s start with a problem that is almost completely unrelated to graphs: the Frankl's union-closed sets conjecture \cite{Frankl}.

\begin{conjecture}[Frankl's union-closed sets conjecture]
If $\mathcal{F}$ is a non-empty union-closed family of subset of $[n]$ (i.e., for all $A,B\in\mathcal{F}$ we have $A\cup B\in\mathcal{F}$), then there exists $k\in[n]$ such that $|\{S\in\mathcal{F}:k\in S\}|\geq|\{S\in\mathcal{F}:k\notin S\}|$
\end{conjecture}

Now we blow-up the union-closed family, but what do we replace an element $k$ with? If we simply use a single set $I_k$ then the new problem is equivalent to the original one, which is not interesting for our purposes. Instead, we should replace $k$ with a family of all non-empty subsets of $I_k$. The reason behind this choice is that it is highly symmetric, and equality in Frankl's union-closed sets conjecture is attained by the family of all subsets of a finite set.

More formally, let $\mathcal{F}\subset2^{[n]}$ be an union-closed family, for each element $k\in[n]$ we take a non-empty set $I_k$ of size $|I_k|=c_k$. Then the resulting union-closed family is $$\mathcal{F}_{(I_k)_{k\in[n]}}=\left\{\bigcup_{k\in S}J_k:S\in\mathcal{F},J_k\subset I_k,J_k\neq\varnothing\right\}$$

We expect the conjecture is true for this new union-closed family, which mean that there exists an element $v_k\in I_k$ such that 

\begin{equation*}
\begin{split}
&|\{S\in\mathcal{F}_{(I_m)_{m\in[n]}}:v_k\in S\}|\geq|\{S\in\mathcal{F}_{(I_m)_{m\in[n]}}:v_k\notin S\}| \\
\Rightarrow &\left|\left\{\bigcup_{m\in S}J_m:S\in\mathcal{F},k\in S,J_m\subset I_m,J_m\neq\varnothing,v_k\in J_k\right\}\right| \\
\geq &\left|\left\{\bigcup_{m\in S}J_m:S\in\mathcal{F},J_m\subset I_m,J_m\neq\varnothing,k\notin S\vee (k\in S\wedge v_k\notin S_k)\right\}\right| \\
\Rightarrow &\sum_{S\in\mathcal{F},k\in S}2^{c_k-1}\prod_{m\in S,m\neq k}(2^{c_m}-1)\geq\sum_{S\in\mathcal{F},k\notin S}\prod_{m\in S}(2^{c_m}-1)+\sum_{S\in\mathcal{F},k\in S}(2^{c_k-1}-1)\prod_{m\in S,m\neq k}(2^{c_m}-1) \\
\Rightarrow &\sum_{S\in\mathcal{F},k\in S}\prod_{m\in S,m\neq k}(2^{c_m}-1)\geq\sum_{S\in\mathcal{F},k\notin S}\prod_{m\in S}(2^{c_m}-1)
\end{split}
\end{equation*}

Let $x_k=2^{c_k}-1$ then $x_k\geq1$ (because $I_k$ is nonempty), the above expression become $$\sum_{S\in\mathcal{F},k\in S}\prod_{m\in S,m\neq k}x_m\geq\sum_{S\in\mathcal{F},k\notin S}\prod_{m\in S}x_m$$

From that we can derive a new conjecture about union-closed set:

\begin{conjecture}
If $\mathcal{F}$ is a non-empty union-closed family of subset of $[n]$ and $x_1,x_2,...,x_n\geq1$ are real numbers, then there exists $k\in[n]$ such that $$\sum_{S\in\mathcal{F},k\in S}\prod_{m\in S,m\neq k}x_m\geq\sum_{S\in\mathcal{F},k\notin S}\prod_{m\in S}x_m$$
\end{conjecture}

We haven't solved Frankl's union-closed sets conjecture yet, but we have discovered a beautiful generalization that may help in solving the original problem.

So, the general way to blow-up an combinatorial object is to replace each element of the object with a structure (usually a set) such that the properties of the object we care are preserved. We then analyze the new object, which may lead to a new generalization and a better understanding of the original object. Unlike other combinatorial techniques such as the probabilistic method, the blow-up trick does not directly help us solve a problem. Instead, it helps us find the right generalization and may also deepen our understanding of the nature of the problem.

\end{document}